 \documentclass[10pt,reqno]{amsart}
  \usepackage{geometry}
\usepackage{amssymb,latexsym}
\usepackage{amsmath}
\usepackage{graphics}

\newtheorem{theorem}{Theorem}[section]
\newtheorem{corollary}[theorem]{Corollary}
\newtheorem{lemma}[theorem]{Lemma}

\newtheorem{defn}[theorem]{Definition}
\newtheorem{proposition}[theorem]{Proposition}

\newcommand{\Hom}{\operatorname {Hom}}
\newcommand{\End}{\operatorname {End}}
\newcommand{\ran}{\mathbf {r}}
\newcommand{\lan}{\mathbf {l}}
\newcommand{\udim}{\operatorname {u-dim}}
\newcommand{\ke}{\operatorname {Ker}}

\begin{document}

\title[Symmetry of extending properties]{Symmetry of extending properties in nonsingular Utumi rings}

%Short paper title for the headers

%% First author ...name + address + email
%% \author{F. Second_name}
%% \address{Street XY, Town, State}
%% \curraddr{...}
%% current address, e-mail and url are optional
%% \email{...@...}
%% \urladdr{...}
\author{Thuat Do}
\address{$^{(1)}$Institute of Research and Development, Duy Tan University, Da Nang 550000,
Vietnam}
\address{$^{(2)}$Department of Science and Technology, Nguyen Tat Thanh University, 300A Nguyen Tat Thanh St, District 4, Ho Chi Minh City, Vietnam}
\email{thuat86@gmail.com}

\author{Hai Dinh Hoang}
\address{International Cooperation Office, Hong Duc University, 565 Quang Trung St, Dong Ve ward, Thanh Hoa city, Vietnam}
\email{hoangdinhhai@hdu.edu.vn}

\author{Truong Dinh Tu$^\dag$}
\address{Faculty of Information Technology, Ton Duc Thang University, Ho Chi Minh City,
Vietnam}
\email{truongdinhtu@tdtu.edu.vn}

%% If there are more authors, then second author, third author contains
%% the same items

%\date{July, 2015}

%% (optional) If any thanks for the financial supports, grants, ...
\thanks{
$^\dag$ The corresponding author, email: truongdinhtu@tdtu.edu.vn}


%% (obligatory) AMS Classification 2000
%% The Primary classification is obligatory,
%% the Secondary classification is optional.

%\subjclass{16D70, 16S50}

 \maketitle

%% (optional) Abstract
\begin{abstract}
This paper presents the right-left symmetry of the CS and max-min CS conditions on nonsingular rings, and generalization to  nonsingular modules. We prove that a ring is right nonsingular right CS and left Utumi if and only if it is left nonsingular left CS and right Utumi. A nonsingular Utumi ring is right max (resp. right min, right max-min) CS if and only if it is left min (resp. left max, left max-min) CS. In addition, a semiprime nonsingular ring is right max-min CS with finite right uniform dimension if and only if it is left max-min CS with finite left uniform dimension.

\end{abstract}

%%%%% private macros, f.e. the different environments

%%% for unnumbered environments, use:
%\newtheorem*{remark}{Remark}
%%% for no-italic, numbered environments, use:
\newenvironment{definition}{\begin{defn}\normalfont}{\end{defn}}

%%%% the main article

\section{Introduction}

Right-left symmetry of extending properties in associative (generally not commutative) rings is extensively studied by many authors. DV. Huynh et al. \cite{GCS} showed that a prime ring is right Goldie right CS with finite right uniform dimension at least two if and only if it is left Goldie left CS with finite left uniform dimension at least two, and a semiprime ring is right Goldie left CS if and only if it is left Goldie, right CS. Later, DV. Huynh \cite{CS1} investigated the symmetry of the CS condition on one-sided ideals in prime rings. SK. Jain et al. \cite{m-mCS} proved the right-left symmetry of the max-min CS property and nonsingularity on prime rings. In more general setting, DV. Thuat et al. \cite{thuat} studied the CS and Goldie conditions in prime and semiprime modules and their endomorphism rings. It is proved that a finite generated, quasi-projective self-generator $M$ is a prime, Goldie and CS module with uniform dimension at least two if and only if its endomorphism ring $S$ is a prime, left Goldie and left CS ring with left uniform dimension at least two; and $S$ is left Goldie and $M$ is CS if and only if $M$ is Goldie and $S$ is left CS. In the mentioned papers, primeness plays an important role to obtain the symmetric properties.  We ask here:

{\em ``If primeness is omitted, can we find some classes of rings in which CS, max CS, min CS and max-min CS properties are right-left symmetric?''}

Firstly, we provide some preliminaries in Section 2. The answer which involves our main results is presented in Section 3. There, the right-left symmetry of the extending properties (we mean the CS, max CS, min CS and max-min CS properties) is proved for the case of associative rings without primeness and even without having finite uniform dimension (see Theorem $\ref{th3.5}$ and Theorem $\ref{th3.7}$). The symmetry of the CS condition on one-sided ideals generated by idempotents is studied in Theorem $\ref{th3.12}$. In addition, the right-left symmetry of the CS, max CS, min CS, max-min CS conditions and finiteness of uniform dimension on nonsingular semiprime rings is shown in Theorem $\ref{thpr}$. Then, we apply the results to the class of nonsingular retractable modules and their endomorphism rings (see Theorem $\ref{th3.10}$, Proposition $\ref{p2}$ and Corollaries $\ref{c1}$, $\ref{c3}$ and $\ref{c4}$). Finally, some examples are discussed to guarantee that our results make sense.

%%%%%%%
\section{Preliminaries}
Throughout this paper, $R$ is an associative (generally not commutative) ring with identity,  $M$ is a unitary right $R-$module with  the endomorphism ring $S=\End(M_R).$ We denote $\ran_X(Y)$ and $\lan_X(Y)$ for the right annihilator and the left annihilator of $Y$ in $X,$ respectively. If there is no chance for misunderstanding of the space $X,$ then we simply write $\ran(Y), \lan(Y)$.

We write $X\hookrightarrow M$ (resp. $X \stackrel{*}{\hookrightarrow}M$) for a submodule (resp. an essential submodule) $X$ of $M.$ A submodule $X$ of $M$ is called a \emph{closed submodule} if $X \stackrel{*}{\hookrightarrow} Y \Rightarrow X=Y,$ for any submodule $Y$ of $M.$ A module $M$ has \emph{finite uniform dimension} if it contains no direct sum of infinitely many nonzero submodules. An $M-$annihilator $X$ of $M_R$ is a submodule provided $X=\ran_M(T)$ for some subset $T$ of $S$. If $M=R,$ then $M-$annihilators are exactly right annihilators of $R$ as usual. A \emph{Goldie module} $M$ is provided that $M$ has finite uniform dimension and $M$ satisfies the ACC (i.e. ascending chain condition) on $M-$annihilators. A \emph{right (left) Goldie} ring  $R$ is provided that $R$ has finite right (left) uniform dimension and $R$ satisfies the ACC on  right (left) annihilators. We denote the uniform dimension of a module $M_R$ by $\udim(M_R).$\\

A {\em CS} (resp. {\em uniform extending}) module is provided that every closed (resp. closed and uniform) submodule is a direct summand. $M$ is called a \emph{max CS module} if every maximal closed submodule with nonzero left annihilator in $S$  is a direct summand. $M$ is called a \emph{min CS module} if every minimal closed submodule is a direct summand. $M$ is called a \emph{max-min CS module} if it is both max CS and min CS.  $R$ is called a \emph{right max CS} (resp. \emph{right min CS, right max-min CS})  ring if $R_R$ is a max CS (resp. min CS, max-min CS)  module. Left max CS, left min CS and left max-min CS ring are defined analogously. It is clear that min CS modules are exactly uniform extending modules. If $M$ has finite uniform dimension, then $M$ is CS if and only if it is min CS. The original notion of right and left max-min CS rings may be seen in \cite{m-mCS}.

The concepts of \emph{nonsingular modules} and \emph{nonsingular rings} are understood as usual. According to \cite{khuri3}, $M$ is  a nonsingular module if and only if for any $X\hookrightarrow M,$ $\ran_R(X)\stackrel{*}{\hookrightarrow} R_R$ implies $X=0.$ $M$ is said to be  \emph{cononsingular}  if  for any $X\hookrightarrow M,$ $\lan_S(X)\stackrel{*}{\hookrightarrow}_SS$ implies $X=0.$ It is equivalent to say that $R$ is right (left) nonsingular if and only if every essential right (left) ideal of $R$ has zero left (right) annihilator. Therefore, $R$ is  right (left) nonsingular if and only if $R_R$ is a nonsingular (cononsingular) module. The following proposition is clear.

\begin{proposition} \label{p0}
The following statements hold for the module $M.$

(1) If $M$ is a nonsingular module, then for any $f\in S, \ke(f) \stackrel{*}{\hookrightarrow} M$ implies $f=0.$ Furthermore, any essential submodule of $M$ has zero left annihilator in $S.$

(2) If $M$ is a cononsingular module, then for any left ideal $K\hookrightarrow S, K \stackrel{*}{\hookrightarrow} _SS$ implies $\ran_M(K)=0.$

\end{proposition}

Now, we consider the converse statements of Proposition $\ref{p0}$. According to \cite{khuri3}, a nonsingular module $M$ is called a \emph{Utumi module}  if every submodule $X$ of $M$ with zero left annihilator in $S$ is essential in $M,$ i.e. $\lan_S(X)=0 \Rightarrow X\stackrel{*}{\hookrightarrow}M.$ A cononsingular module $M$ is  called a \emph{co-Utumi module} if every left ideal $K$ of $S$ with zero right annihilator in $M$ is essential in $S,$ that is  $\ran_M(K)=0 \Rightarrow K\stackrel{*}{\hookrightarrow} _SS.$ $R$ is called a \emph{right (left) Utumi ring} if $R_R$ is a Utumi (co-Utumi) module. By a nonsingular (Utumi) ring we mean that it is right and left nonsingular (Utumi). The two following lemmas are easy.

%3.3
\begin{lemma} \label{lm3.3}

If $M$ is a CS module, then $M$ is Utumi. In particular, a right CS ring is right Utumi.
\end{lemma}

%3.4
\begin{lemma}\label{lm3.4}

If $M$ is a nonsingular CS module, then $M$ is cononsingular. In particular, a right nonsingular right CS ring is left nonsingular.
\end{lemma}
%see more at last paragraph in \cite{khuri3}, but need $M$ to be retractable. Then $S$ is Baer so is left nonsingular.
%Let $R$ be right nonsingular right CS. For any subset $A\subet R, \ran_R(A)$ is essential in a direct summand $eR, e=e^2.$ We have $A\subset \lan\ran(A)=R(1-e)$ and $eR\subset\ran(A)$ so $\ran(A)=eR.$ This means that $R$ is a Baer ring so is left nonsingular.

For a submodule $X$ of $M,$ we write $I_X:=\{f\in S|f(M)\subseteq X\}.$ For a subset $K$ of $S,$ we write $KM=K(M):=\sum\limits_{f\in K}f(M).$ It is clear that $I_X$ is a right ideal of $S$ and $KM$ is a submodule of $M.$ The two following conditions  are introduced and investigated in  \cite{khuri-nd, khuri4}.

\begin{itemize}
  \item (I) For submodules $X, Y$ of $M, X \stackrel{*}{\hookrightarrow} Y$ if and only if $I_X\stackrel{*}{\hookrightarrow}I_Y.$
  \item (II) For right ideals $K, L$ of $S, K \stackrel{*}{\hookrightarrow} L$ if and only if $KM\stackrel{*}{\hookrightarrow}LM.$
\end{itemize}

We observe that every finitely generated, quasi-projective self-generator is retractable and it possesses (I) and (II) (see \cite[Lemma 2.2]{thuat}). The same assertion holds for nondegenerate modules (see \cite{khuri-nd}). In the following lemma, we sum up \cite[Theorem 2.2]{khuri4} and \cite[Theorem 2.5]{khuri4} to make a tool to prove our main results in the
subsequent section.

\begin{lemma} (see \cite{khuri4}) \label{lm3.9}

  (1) If $M$ is a nonsingular and retractable module, then (I) holds.

  (2) Given the condition (I), then the condition (II) holds if and only if $K\stackrel{*}{\hookrightarrow}I_{KM}$ for every right ideal $K\hookrightarrow S_S.$

  (3) Given the condition (II), then the condition (I) holds if and only if $I_X(M)\stackrel{*}{\hookrightarrow}X$ for every submodule $X\hookrightarrow M.$
\end{lemma}

%%%%%%%%%%%%
\section{The main results}

We agree an abbreviation that MRQR and MLQR indicate maximal right quotient ring and maximal left quotient ring, successively. According to \cite{khuri4}, $M$ is a {\em retractable module} if and only if $\Hom(M, X)\neq 0$ for every $0\neq X\hookrightarrow M.$ We denote the injective hull (or the envelope) of $M$ by $E(M),$ and the endomorphism ring of $E(M)$ by $T=\End_R(E(M)).$ The following lemma plays an important role in our investigation.

%3.2
\begin{lemma} \label{lm3.2}  \cite[Theorem 2]{khuri3}

Let $M$ be a retractable, nonsingular and cononsingular module. Then, $T=\End_R(E(M))$  is both the MRQR and the MLQR of $S=\End(M_R)$ if and only if $M$ is both a Utumi and co-Utumi module. In particular, for a nonsingular ring $R,$ the MRQR and the MLQR of $R$ coincide if and only if $R$ is right and left Utumi.
\end{lemma}

Note that in the case of Lemma $\ref{lm3.2}$, if $Q$ is the MRQR and the MLQR of $R$, then $Q$ is also the injective hull of $R_R$ and $_RR$. Therefore, $Q$ is von Neumann regular, right and left self-injective. Moreover, by \cite[Lemma 1.4]{utumi}, $Q$ can be regarded as the ring consisting of element $x$ such that the set of $y\in R$ with $xy\in R$ forms an essential right ideal of $R.$ This notation will serve us in proof of subsequent theorems. Under the aid of the conditions (I) and (II), we derive the following results.

\begin{lemma} \label{lm3.13}
 Let $M$ be a retractable module which possesses (I) and (II). Then, $\udim(M_R)=n$ if and only if $\udim(S_S)=n$, where $n\geq 0$ is an integer.
\end{lemma}
% (I) and (II) implies retractability
\begin{proof}
We give some observations before mutually converting finiteness of uniform dimension between $M_R$ and $S_S.$ For nonzero right ideals $K, H$ of $S,$ we have $K(M)\neq 0$ and $H(M)\neq 0$. Moreover, we claim that $K\cap H=0$ if and only if $K(M)\cap H(M)=0$. We assume that $K\cap H=0$ and $K(M)\cap H(M)=Y\neq 0$. Then, by retractability of $M,$ there exists $0\neq s\in \Hom(M, Y)$, whence $s\in I_{KM}\cap I_{HM}$. By Lemma $\ref{lm3.9}$, we have $K\overset{*}{\hookrightarrow}I_{KM}$ and $H\overset{*}{\hookrightarrow}I_{HM}$. Therefore, there exists $f\in S$ such that $0\neq sf\in K$ and $sfS\cap H\neq 0.$ This means $K\cap H\neq 0,$ a contradiction. Thus, we must have $K(M)\cap H(M)=0.$ The converse is proved similarly.

For nonzero submodules $A$ and $B$ of $M,$ we see that $I_A\neq 0, I_B\neq 0$. Moreover, $A\cap B=0$ if and only if $I_A\cap I_B=0$ because of retractability of $M$. If we have $A\oplus B$, then we also get $I_A\oplus I_B$. It is obvious that $(I_A\oplus I_B)(M)\subseteq I_A(M)\oplus I_B(M)$. For $z=a+b\in I_A(M)\oplus I_B(M),$ where $a\in I_A(M), b\in I_B(M),$ there are $f_A\in I_A, f_B\in I_B$ and $x, y\in M$ such that $a=f_A(x), b=f_B(y).$ We see that $a\in f_A(M)\subset (I_A\oplus I_B)(M)$ and $b\in f_B(M)\subset (I_A\oplus I_B)(M)$. This implies $a+b\in (I_A\oplus I_B)(M)$, whence $I_A(M)\oplus I_B(M)\subseteq (I_A\oplus I_B)(M)$. Therefore, we get $(I_A\oplus I_B)(M)=I_A(M)\oplus I_B(M)$.

Now, let $A$ and $B$ be submodules of $M$ with $A\oplus B\overset{*}{\hookrightarrow}M$. Then, we have $(I_A\oplus I_B)(M)=I_A(M)\oplus I_B(M)$. On the other hand, by Lemma $\ref{lm3.9}, I_A(M)\oplus I_B(M)\overset{*}{\hookrightarrow}A\oplus B$ and $I_A\oplus I_B\overset{*}{\hookrightarrow}I_{A\oplus B}\overset{*}{\hookrightarrow}S_S$. Similarly, for right ideals $K, H$ of $S,$ if $K\oplus H\overset{*}{\hookrightarrow} S_S,$ then $(K\oplus H)(M)\overset{*}{\hookrightarrow} M_R$ and hence $K(M)\oplus H(M)\overset{*}{\hookrightarrow} M_R$. By these arguments, we inductively induce that for any integer $n\geq 0, \udim(M_R)=n$ if and only if $\udim(S_S)=n$.
\end{proof}

\begin{proposition} \label{p1}
  Let $M$ be a nonsingular and co-nonsingular, Utumi and co-Utumi retractable module which possesses (II). Then, we have $\udim(M_R)=n$ if and only if $\udim(S_S)=n$ if and only if $\udim(_SS)=n$, where $n\geq 0$ is an integer. In this case, $M_R, S_S$ and $_SS$ are Goldie modules.

  In particular, let $R$ be a nonsingular Utumi ring. Then, $\udim(R_R)=n$ if and only if $\udim(_RR)=n$. In this case, $R$ is right and left Goldie.
\end{proposition}

\begin{proof}
Since $M$ is nonsingular and retractable, Lemma $\ref{lm3.9}$ asserts that $M$ possesses (I). By Lemma $\ref{lm3.13}$, we have $\udim(M_R)=n=\udim(S_S).$ By \cite[Theorem 3.1]{khuri}, nonsingularity of $M$ implies right nonsingularity of $S$. Since $M$ is co-nonsingular, by \cite[Proposition 1]{khuri3}, $S$ is left nonsingular, so it is nonsingular. By Lemma $\ref{lm3.2}$,  $T$ is both the MRQR and the MLQR of $S.$ Thus, we have $n=\udim(S_S)=\udim(T_S)=\udim(_ST)=\udim(_SS).$

Since $M$ is nonsingular module with finite uniform dimension, $M$ satisfies the ACC on $M-$annihilators. Thus, $M$ is a Goldie module. Since $S$ is nonsingular with finite right and left uniform dimensions, $S$ satisfies the ACC on right and left annihilators. Thus, $S$ is right and left Goldie.

\end{proof}

%3.5
\begin{theorem} \label{th3.5}
The following statements are equivalent for a ring $R.$

(1) $R$ is a right nonsingular, right CS and left Utumi ring;

(2) $R$ is a left nonsingular, left CS and right Utumi ring.

In this case, if either $R_R$ or $_RR$ has finite uniform dimension, then $R$ is a right and left Goldie ring.
\end{theorem}

\begin{proof}
We assume that $R$ is a right nonsingular, right CS and left Utumi ring. By Lemma $\ref{lm3.3}$, $R$ is right Utumi, so it is Utumi. By Lemma $\ref{lm3.4}$, $R$ is left nonsingular, so  it is nonsingular. Since $R$ is a nonsingular, Utumi ring, the MRQR and the MLQR of $R$ coincide by Lemma $\ref{lm3.2}$, and denoted by $Q.$

Now, we prove that $R$ is left CS. For any closed left ideal $I$ of $R,$ by the lattice isomorphism \cite[Corollary 2.6]{john}, we have $I=J\cap R$ for some closed left ideal $J$ of $Q.$ Then, $J$ is a direct summand of $Q,$ writing $J=Qe$ for some idempotent $e\in Q.$ We easily see that $\ran_Q(e)=(1-e)Q$ is a closed right ideal of $Q,$ thus $(1-e)Q\cap R$ is a closed right ideal of $R.$ Since $R$ is right CS, we get $(1-e)Q\cap R=fR$ for some $f=f^2\in R.$ We set $K=\{k\in R|(1-e)k\in R\}.$ Then, $K$ is an essential right ideal of $R.$ We have $R(1-f)=\lan_R(fR)=\lan_R[(1-e)Q\cap R]=\lan_R[(1-e)K]=\{x\in R|x(1-e)K=0\}=\{x\in R|x(1-e)=0\}=\lan_Q(1-e)\cap R=Qe\cap R=I.$ Thus, $I$ is a direct summand of $R.$ This implies that $R$ is left CS.

The converse is right-left symmetric.

The last statement is referred to Proposition $\ref{p1}.$
\end{proof}

%3.6
\begin{corollary}
A right nonsingular, right CS and left Utumi ring is directly finite.
\end{corollary}

\begin{proof}
  It follows from Theorem $\ref{th3.5}$ and the fact that a right and left CS ring is directly finite.
\end{proof}

%3.8
\begin{corollary} \label{c1}
If $M$ is a  nonsingular, retractable module, then the following statements are equivalent:

(1) $M$ is a co-Utumi, CS module;

(2) $S$ is a left Utumi, right CS ring;

(3) $S$ is a right Utumi, left CS ring.

In addition, if $M$ has finite uniform dimension, then $M_R, S_S$ and $_SS$ are Goldie modules, and $\udim(M_R)=\udim(S_S)=\udim(_SS).$
\end{corollary}

\begin{proof} We observe that $S$ is right nonsingular by \cite[Theorem 3.1]{khuri}.

{\em (1) $\Leftrightarrow$ (2)} Since $M$ is nonsingular and retractable, by \cite[Theorem 3.2]{khuri4}, $M$ is CS if and only if $S$ if right CS.  Since $M$ is nonsingular and CS, $M$ is a Utumi and co-nonsingular module by Lemma $\ref{lm3.3}$ and Lemma $\ref{lm3.4}$, respectively.  Therefore,  $S$ is left Utumi if and only if $M$ is co-Utumi by \cite[Lemma 4]{khuri3}.

{\em (2) $\Leftrightarrow$ (3)} It follows from Theorem $\ref{th3.5}$.

See Proposition $\ref{p1}$ for the last statement.

\end{proof}

%3.7
\begin{theorem} \label{th3.7}
The following statements hold for every nonsingular Utumi ring $R.$

(1) $R$ is right min CS if and only if $R$ is left max CS.

(2) $R$ is right max CS if and only if $R$ is left min CS.

(3) $R$ is  right max-min CS if and only if $R$ is  left max-min CS.

\end{theorem}

\begin{proof} Since $R$ is nonsingular Utumi, the MRQR and the MLQR of $R$ coincide by Lemma $\ref{lm3.2}$, and denoted by $Q.$

\emph{(1)} We assume that $R$ is right min CS. For any maximal closed left ideal $I$ of $R$ with $\ran_R(I)\neq 0,$ by the lattice isomorphism \cite[Corollary 2.6]{john}, we have $I=J\cap R$ for some closed left ideal $J$ of $Q.$ If $J$ is contained in some closed left ideal $K$ of $Q,$ then $K\cap R$ is a closed left ideal of $R$ and $I\subseteq K\cap R.$ Since $I$ is maximal closed, $I=K\cap R=J\cap R$ so $K=J.$ This shows that $J$ is a maximal closed left ideal of $Q.$ It is clear that $J$ is a direct summand of $Q,$ so $J=Qe$ for some idempotent $e\in Q.$ We easily see that $\ran_Q(Qe)=\ran_Q(e)=(1-e)Q$ is a closed right ideal of $Q,$ thus $(1-e)Q\cap R$ is a closed right ideal of $R.$ We will show that $(1-e)Q$ is minimal closed in $Q.$ Suppose that $H=tQ, t=t^2\in Q,$ is a closed right ideal of $Q$ such that $H\subseteq (1-e)Q.$ Then, $\lan_Q(H)=Q(1-t)\supseteq \lan_Q[(1-e)Q]=Qe.$ Since $Qe=J$ is maximal closed in $Q,$ $Qe=Q(1-t)$ and hence $(1-e)Q=\ran_Q(Qe)=\ran_Q[Q(1-t)]=tQ=H.$ This implies that $(1-e)Q$ is minimal closed in $Q$. Thus $(1-e)Q\cap R$ is a minimal closed right ideal of $R.$ Since $R$ is right min CS, $(1-e)Q\cap R=fR$ for some idempotent $f\in R.$ We set $F=\{k\in R|(1-e)k\in R\}.$ Then, $F$ is an essential right ideal of $R.$ We have $R(1-f)=\lan_R(fR)=\lan_R[(1-e)Q\cap R]=\lan_R[(1-e)F]=\{x\in R|x(1-e)F=0\}=\{x\in R|x(1-e)=0\}=\lan_Q(1-e)\cap R=Qe\cap R=I.$ Thus, $I$ is a direct summand of $R.$ This shows that $R$ is left max CS.

Conversely, let $R$ be a left max CS ring. For any properly minimal closed right ideal $I$ of $R,$ by the lattice isomorphism \cite[Corollary 2.6]{john}, we have $I=J\cap R$ for some closed right ideal $J$ of $Q.$ If $J$ contains a closed right ideal $K$ of $Q,$ then $K\cap R$ is a closed right ideal of $R$ and $(K\cap R)\subseteq (J\cap R)=I.$ Since $I$ is minimal closed, $I=K\cap R=J\cap R$ so $K=J.$ This shows that $J$ is a minimal closed right ideal of $Q.$ We write $J=eQ$ for some idempotent $0\neq e\in Q.$ We observe that $\lan_Q(eQ)=\lan_Q(e)=Q(1-e)$ is a closed left ideal of $Q,$ thus $Q(1-e)\cap R$ is a closed left ideal of $R.$ We will prove that $Q(1-e)$ is maximal closed in $Q.$ Suppose that $H=Qt, t=t^2\in Q,$ is a closed left ideal of $Q$ such that $H\supseteq Q(1-e).$ Then, $\ran_Q(H)=(1-t)Q\subseteq \ran_Q[Q(1-e)]=eQ.$ Since $eQ=J$ is minimal closed in $Q,$ $eQ=(1-t)Q$ and hence $Q(1-e)=\lan_Q(eQ)=\lan_Q[(1-t)Q]=Qt=H.$ This implies that $Q(1-e)$ is maximal closed in $Q,$ thus $Q(1-e)\cap R$ is a maximal closed left ideal of $R.$

Because of $e\neq0,$ we have $0\neq eQ\cap R\subset \ran_R[Q(1-e)\cap R]$. Since $R$ is left max CS, $Q(1-e)\cap R=Rf$ for some idempotent $f\in R.$ We set $F=\{k\in R|k(1-e)\in R\}.$ Then, $F$ is an essential left ideal of $R.$ We have $(1-f)R=\ran_R(Rf)=\ran_R[Q(1-e)\cap R]=\ran_R[F(1-e)]=\{x\in R|F(1-e)x=0\}=\{x\in R|(1-e)x=0\}=\ran_Q(1-e)\cap R=eQ\cap R=I.$ Thus, $I$ is a direct summand of $R.$ This shows that $R$ is right min CS.

\emph{(2)} It is dual to the proof of \emph{(1)}.

\emph{(3)} It is induced from \emph{(1)}  and \emph{(2)}.

\end{proof}

By \cite[Theorem 3.2]{khuri4}, a nonsingular retractable module is CS if and only if its endomorphism ring is right CS. We wish to find an analogue for the max-min CS property. With the aid of (I) and (II), we will transfer the max CS, min CS and max-min CS properties of a module to its endomorphism in the next theorem.

%3.9
\begin{theorem} \label{th3.10}
Let $M$ be a nonsingular and retractable module which possesses the condition (II). Then, the following statements hold.

(1) $M$ is min CS if and only if $S$ is right min CS.

(2) $M$ is max CS if and only if $S$ is right max CS.

(3) $M$ is  max-min CS if and only if $S$ is  right max-min CS.

\end{theorem}

\begin{proof}
It is clear that {\em (3)} follows from {\em (1)} and {\em (2)}. Note that since $M$ is a nonsingular module, every submodule $X$ has a unique closure (i.e. there is a unique closed submodule of $M$ that essentially contains $X$).

{\em (1)} Let $M$ be a min CS module. For a minimal closed (or uniform closed) right ideal $K$ of $S,$ we have  $K\stackrel{*}{\hookrightarrow}I_{KM}$ by Lemma $\ref{lm3.9}$ so $K=I_{KM}$. For nonzero submodules $U, V$ of $KM,$ since $M$ is retractable, $I_U$ and $I_V$ are nonzero. It is clear that $I_U$ and $I_V$ are contained in $I_{KM}=K,$ thus $I_U \cap I_V\neq 0.$ Then, there exists $0\neq s \in I_U \cap I_V,$ whence $0\neq s(M)\subset U\cap V.$ Therefore, $KM$ is uniform. Since $M$ is min CS, $KM$ is essential in a direct summand of $M,$ namely $X.$ We have $X=e(M)$ for some $e=e^2\in S$ and by the condition (I), $K=I_{KM}\stackrel{*}{\hookrightarrow}I_{e(M)}=eS.$ Closeness of $K$ implies that $K=eS$, and hence $K$ is a direct summand of $S.$ Consequently, $S$ is right min CS.

Conversely, let $S$ be a right min CS ring. For a uniform closed submodule $X$ of $M,$ $I_X$ is a right ideal of $S.$ If arbitrary nonzero right ideals $K, L$ are contained in $I_X,$ then $KM, LM$ are nonzero submodules contained in $X.$ Since $X$ is uniform, $KM\cap LM=Y\neq 0$. By retractability of $M,$ there exists $0\neq s\in S$ such that $s(M)\subset Y.$ Therefore, we get $s\in I_{KM}\cap I_{LM}.$ On the other hand, $K \stackrel{*}{\hookrightarrow}I_{KM}$ and $L\stackrel{*}{\hookrightarrow}I_{LM}$ follows from Lemma $\ref{lm3.9}.$ Thus, there exists $f\in S$ such that $0\neq sf\in K,$ and $sfS\cap L\neq 0.$ This implies that $K\cap L\neq 0$ so $I_X$ is uniform. Since $S$ is right min CS, $I_X$ is essential in a direct summand $J=eS$ for some $e=e^2\in S.$ Then, by the condition (II), $I_X(M)$ is essential in $eS(M)=e(M)$, a direct summand of $M.$ By Lemma $\ref{lm3.9},$ $I_X(M)$ is essential in $X.$ But $I_X(M)$ has one closure only. Therefore, we must have $X=e(M).$ This shows that $M$ is min CS.

{\em (2)} Let $M$ be a max CS module. For a maximal closed right ideal $K$ of $S$ with $\lan_S(K)\neq 0,$ we have $K=I_{KM}$ as arguing in {\em (1)}. It is induced from the condition (II) that $KM$ is not essential in $M$, since $K$ is not essential in $S$. Thus, there exists a maximal closed submodule $X\hookrightarrow M$ containing $KM$ and $X\neq M$. We have $K=I_{KM}\subset I_X$ so $K=I_X$ by maximality of $K$. This implies that $KM=I_X(M)\stackrel{*}{\hookrightarrow}X$.
Because  $\lan_S(I_X)=\lan_S(K)\neq 0$, there is some $f\in S$ so that $f(KM)=0.$ By Proposition $\ref{p0}$, we also have $f(X)=0$ so $\lan_S(X)\neq 0$. Since $M$ is max CS, $X$ is a direct summand of $M,$ writing $X=e(M)$ for some $e=e^2\in S.$ Then, we have $K\subset I_{e(M)}=eS$. Since $K$ is maximal closed, $K=eS$ holds true. Thus, $S$ is right max CS.

Conversely, let $S$ be a right max CS ring. For a maximal closed submodule $X$ of $M$ with nonzero left annihilator in $S,$ we have $I_X(M)\stackrel{*}{\hookrightarrow}X$ and $0\neq\lan_S(X)\subset\lan_S(I_X)$. By the condition (I), $I_X$ is not essential in $S$, since $X$ is not essential in $M$. Thus, there exists a maximal closed right ideal $K\hookrightarrow S$ containing $I_X$ and $K\neq S.$ We observe that $I_X(M)\subset KM.$ On the other hand, $I_X(M)$ has a unique maximal essential extension, so $KM\subset X$ because of maximality of $X$. This shows that $K=I_X$ and hence $I_X=eS$ for some $e=e^2\in S,$ since $S$ is right max CS.  Therefore, we get $I_X(M)\hookrightarrow e(M),$ whence $X=eM,$ a direct summand of $M.$ This proves that $M$ is max CS.

\end{proof}

By Theorem $\ref{th3.7}$ and Theorem $\ref{th3.10}$, we do have.

\begin{corollary} \label{c3}

Let $M$ be a  retractable, nonsingular and co-nonsingular, Utumi and co-Utumi  $R-$module which possesses the condition (II). Then, the following statements hold.

(1) $M$ is min CS if and only if $S$ is right min CS if and only if $S$ is left max CS.

(2) $M$ is max CS if and only if $S$ is right max CS if and only if $S$ is left min CS.

(3) $M$ is  max-min CS if and only if $S$ is  right max-min CS if and only if $S$ is left max-min CS.

\end{corollary}

Motivated by \cite[Theorem 3]{CS1}, we study the symmetry of the CS property on one sided-ideals in the following theorem.

\begin{theorem} \label{th3.12}
Let $R$ be a nonsingular Utumi ring. Then, the following conditions are equivalent for every $e=e^2\in R.$

(1) $eR_R$ is CS with finite uniform dimension;

(2) $_RRe$ is CS with finite uniform dimension;

(3) $eRe$ is right  CS with finite right uniform dimension;

(4) $eRe$ is left CS with finite left uniform dimension.

In this case, $eR_R$ and $_RRe$ are Goldie modules, $eRe$ is a right and left Goldie ring, and $\udim(eR_R)=\udim(_RRe)=\udim(eRe_{eRe})=\udim(_{eRe}eRe)$.

\end{theorem}
\begin{proof}
Since $R$ is nonsingular Utumi, the MRQR and MLQR of $R$ coincide by Lemma $\ref{lm3.2}$, denoted by $Q$.

\emph{(1)$\Leftrightarrow$(3)} Let $eR_R$ be a CS module with finite uniform dimension. Then, $eQ$, the injective hall of $eR,$ is a semisimple artinian right ideal of $Q.$ Furthermore, because $\End(eQ)\cong eQe$, we see that $eQe$ is a semisimle artinian ring which is the MRQR of $eRe$. Thus, $eRe$ has  finite right uniform dimension. In order to show that $eRe$ is right CS, it is sufficient to prove that every uniform closed right ideal $V$ of $eRe$ is a direct summand of $eRe.$ Clearly, $V=f(eQe)\cap eRe$ for some $f=f^2\in eQe.$ We observe that $f=ef=fe$ and $fQ\hookrightarrow eQ.$ Thus, $fQ\cap R$ is closed in $R$ and contained in $eR.$ Therefore, $fQ\cap R$ is closed in $eR$. Since $eR_R$ is a CS module, $fQ\cap R$ is a direct summand of $eR$ so of $R$. This means $fQ\cap R=gR$ for some idempotent $g\in R.$ We have $ge=ege=(ege)^2.$ Hence, $ege(eRe)=geRe$ is a direct summand of $eRe$ which is also contained in $V.$ Since $V$ is minimal closed, $V=geRe$ is a direct summand of $eRe.$ This implies that $eRe$ is a right CS ring.

Conversely, let $eRe$ be a right CS ring with finite right uniform dimension. Then, $eQe$ is a semisimple artinian ring and $eQ_R$, the injective hall of $eR_R$, is a semisimple artinian and noetherian module. Thus, $eR_R$ has finite uniform dimension. Let $V$ be a minimal closed submodule of $eR.$ We have $V=fQ\cap R$, where $f=f^2\in Q$ and $fQ\hookrightarrow eQ.$ We observe that $fQ$ of a simple component of $eQ.$ Thus, $fQe$ is a simple component of $eQe$. Therefore, $fQe\cap eRe$ is a minimal closed right ideal of $eRe,$ hence $fQe\cap eRe=g(eRe)$, where $g=g^2\in eRe.$ We see that $gR$ is minimal closed. Because of $Ve\overset{*}{\hookrightarrow}fQ, Ve\cap gRe\neq 0$ and $V\cap gR\neq 0, V$ and $gR$ are both the unique closure of $V\cap gR$, whence $V=gR$. This shows that $V$ is a direct summand of $eR$ and hence $eR_R$ is CS.

\emph{(2)$\Leftrightarrow$(4)} We argue similarly to \emph{(1)$\Leftrightarrow$(3)}.

\emph{(3)$\Leftrightarrow$(4)} Let $eRe$ be a right  CS ring with finite right uniform dimension. Then, $eQe$, the MRQR of $eRe,$ is semisimple artinian and is also the MLQR of $eRe.$ Therefore, $eRe$ has finite left uniform dimension and is left CS (see proof of Theorem \ref{th3.5}). The converse is symmetric.

The last statement is referred to Proposition $\ref{p1}.$

\end{proof}

%3.3
\begin{proposition} \label{p2}

Let $M$ be a nonsingular retractable $R-$module which possesses (II). Then, the following statements hold for any $e=e^2\in S.$

(1) $e(M)$ is CS if and only if $eS_S$ is CS.

(2) $e(M)$ is min CS if and only if $eS_S$ is min CS.

(3) $e(M)$ is max CS if and only if $eS_S$ is max CS.

(4) $e(M)$ is max-min CS if and only if $eS_S$ is max-min CS.
\end{proposition}

\begin{proof}
We argue similarly to the proof of Theorem $\ref{th3.10}$. Note that if $K$ is a closed right ideal of $S$ contained in $eS$, then $KM$ is contained in $e(M).$ Conversely, if $Y$ is a closed submodule of $M$ contained in $e(M),$ then $I_Y$ is a right ideal of $S$ contained in $I_{e(M)}=eS.$

\end{proof}

By Theorem $\ref{th3.12}$ and Proposition $\ref{p2}$, we have.

\begin{corollary} \label{c4}
Let $M$ be a  retractable $R-$module which possesses (II). If $M$ is nonsingular and co-nonsingular, Utumi and co-Utumi, then the following conditions are equivalent for every $e=e^2\in S.$

  (1) $eM$ is CS with finite uniform dimension;

  (2) $eS_S$ is CS with finite uniform dimension;

  (3) $_SSe$ is CS with finite uniform dimension;

  (4) $eSe$ is right  CS with finite right uniform dimension;

  (5) $eSe$ is left CS with finite left uniform dimension.

\end{corollary}

As we mentioned in the introduction, this paper mainly consider rings without primeness. However, the following theorem give us an additional symmetry of the extending properties and finiteness of uniform dimension on nonsingular semiprime rings. This is not investigated in \cite{GCS, CS1, m-mCS}.

\begin{theorem} \label{thpr}
Let $R$ be a semiprime ring.

(1) $R$ is right CS, right nonsingular with finite right uniform dimension if and only if $R$ is left CS, left nonsingular with finite left uniform dimension.

(2) If $R$ is nonsingular, then the following statements hold true.
\begin{itemize}
  \item (2.1) $R_R$ is max CS with finite uniform dimension if and only if $_RR$ is min CS with finite uniform dimension.
  \item (2.2) $R_R$ is min CS with finite uniform dimension if and only if $_RR$ is max CS with finite uniform dimension.
  \item (2.3) $R_R$ is max-min CS with finite uniform dimension if and only if $_RR$ is max-min CS with finite uniform dimension.
\end{itemize}

In all the cases above, $R$ is right and left Goldie with $\udim(R_R)=\udim(_RR)=n$ for some integer $n\geq 0.$
\end{theorem}

\begin{proof} For the case of \emph{(1)}, Lemma $\ref{lm3.4}$ implies that a right CS right nonsingular ring is left nonsingular, and a left CS left nonsingular ring is right nonsingular. Thus, $R$ is right and left nonsingular for both cases \emph{(1)} and \emph{(2)}.

  Since $R$ is nonsingular, $R$ has a maximal two-sided quotient ring $Q$ by \cite[Lemma 1.4]{utumi}. Since $R_R$ has finite uniform dimension, $Q$ is semisimple.
  %by [Theorem 1.6, SEMISIMPLE MAXIMAL QUOTIENT  RINGS, by FRANCIS L. SANDOMIERSKK].
  Therefore, $\udim(_RQ)$ is finite so is $\udim(_RR)$. Since $R$ is nonsingular with finite right and left uniform dimension, $R$ is a right and left Goldie ring by
  \cite[Corollary 3.32]{Go}. Therefore, $Q$ is a classical right and left quotient ring of $R$ as well as a maximal right and left quotient ring of $R$ by \cite[Theorem 3.37]{Go}. We argue similarly when $_RR$ has finite uniform dimension.

Now, the proof about equivalence of the extending conditions on the right and left sides of $R$ is similar to Theorem $\ref{th3.5}$ and Theorem $\ref{th3.7}$.
\end{proof}

{\bf Examples.}
It is easy to find examples of right and left max-min CS rings. In particular, one of such a ring is $R=\left(
                                       \begin{array}{cc}
                                         F & F \\
                                         0 & F \\
                                       \end{array}
                                     \right),$
                                     where $F$ is a field.

There is a module which is neither max CS nor min CS. Let  $Z$ be the set of all integers. Consider $Z-$module $M=(Z/Z2)\oplus (Z/Z8).$ We observe that $A=((1+Z2)\oplus (2+Z8))Z$ is a minimal closed submodule of $M$ but not a direct summand. It is easy to verify that $A$ is also a maximal closed submodule with non-zero left annihilator in the endomorphism ring of $M$. Thus, $M$ is neither max CS nor min CS, although $M$ has finite uniform dimension.

There exists a ring which is a right Ore domain but not a left Ore domain. Such a ring is mentioned (namely $R$) in  \cite[Exercise 1, page 101]{Go}. It is not difficult to see  that $R$ is right max-min CS but not left min CS. If $R$ is min CS, then $R$ must be left uniform, so is left Ore, a contradiction.


\begin{thebibliography}{10}

\bibitem{Go} K.R. Goodearl. {\em Ring Theory: Nonsingular Rings and Modules}, CRC Press, 1976.

\bibitem{GCS} D. V. Huynh, S. K. Jain, and S. R. Lopez-Permouth. On the symmetry of the Goldie and CS conditions for prime rings.  {\em Proceedings of the American Math. Soc.}, {\bf 128:11}(2000), 3153-3157.

\bibitem{CS1} D. V. Huynh. The symmetry of the CS condition on one-sided ideals in a prime ring.  {\em J. Pure and Applied Algebra}, {\bf 212}(2008), 9-13.

\bibitem{m-mCS} S. K. Jain, Husain S. Al-Hazmi, and Adel N. Alahmadi. Right-Left Symmetry of Right Nonsingular Right Max-Min CS Prime Rings.  {\em Communications in Algebra}, {\bf 34}(2006), 3883-3889.

\bibitem{john} R. E. Johnson. Quotient rings with zero singular ideal.   {\em Pacific J. Math}, {\bf 11}(1961),  1358-1392.

\bibitem{khuri} S. M. Khuri. Endomorphism rings of nonsingular modules. {\em  Ann. Sci. Math. Quebec}, {\bf 4}(1980), 145-152.

\bibitem{khuri3} S. M. Khuri. Modules whose endomorphism rings have isomorphic maximal left and right quotient rings. {\em  Proceedings of the American Math. Soc.}, {\bf 85:2}(1982), 161-164.

\bibitem{khuri-nd} S. M. Khuri. Correspondence theorems for modules and their endomorphism rings. {\em  J. Algebra}, {\bf 122}(1989), 380-396.

\bibitem{khuri4} S. M. Khuri. Nonsingular retractable modules and their endomorphism rings. {\em  Bull. Austral. Math. Soc.},  {\bf 43}(1991),  63-71.

\bibitem{thuat} D. V. Thuat, H. D. Hai and N. V. Sanh. On Goldie prime CS-modules. {\em East-West J. Math.}, {\bf 16:2}(2014), 131-140.


\bibitem{utumi} Utumi, Y. On prime J-rings with uniform one-sided ideals. {\em Amer. J. Math.}, {\bf 85:4}(1963), 583-596.


\end{thebibliography}
\end{document}